\documentclass{amsart}
\usepackage{amsmath}
\usepackage{amsfonts}
\usepackage{amsthm, upref}
\usepackage{graphicx}
\usepackage[usenames, dvipsnames]{color}

\newcommand{\comment}[1]{}
\newcommand{\eq}{\begin{equation}}
\newcommand{\en}{\end{equation}}

\newcommand{\NN}{\mathbb{N}}

\newcommand{\abs}[1]{\left\lvert #1 \right\rvert}

\newcommand{\combi}[2]{\begin{pmatrix}#1 \\ #2 \end{pmatrix}}

\newcommand{\bridge}{\text{Bridge}}
\newcommand{\exc}{\mathrm{Exc}}

\begin{document}

\theoremstyle{plain}
\newtheorem{thm}{Theorem}
\newtheorem{lemma}[thm]{Lemma}
\newtheorem{prop}[thm]{Proposition}
\newtheorem{cor}[thm]{Corollary}

\theoremstyle{definition}
\newtheorem{defn}{Definition}
\newtheorem{asmp}{Assumption}
\newtheorem{notn}{Notation}
\newtheorem{prb}{Problem}

\theoremstyle{remark}
\newtheorem{rmk}{Remark}
\newtheorem{exm}{Example}
\newtheorem{clm}{Claim}

\title[Brownian approximation]{Brownian approximation to counting graphs}

\author{Soumik Pal}
\address{Department of Mathematics\\ University of Washington\\ Seattle, WA 98195}
\email{soumik@u.washington.edu}

\keywords{Wright's formula, Area under Brownian excursion, counting graphs}

\subjclass[2000]{}

\thanks{This research is partially supported by NSF grant DMS-1007563}

\date{\today}

\begin{abstract} 
Let $C(n,k)$ denote the number of connected graphs with $n$ labeled vertices and $n+k-1$ edges. For any sequence $(k_n)$, the limit of $C(n,k_n)$ as $n$ tends to infinity is known. It has been observed that, if $k_n=o(\sqrt{n})$, this limit is asymptotically equal to the $k_n$th moment of the area under the standard Brownian excursion. These moments have been computed in the literature via independent methods. In this article we show why this is true for $k_n=o(\sqrt[3]{n})$ starting from an observation made by Joel Spencer. The elementary argument uses a result about strong embedding of the Uniform empirical process in the Brownian bridge proved by Koml\'os, Major, and Tusn\'ady.  
\end{abstract}

\maketitle

\section{Introduction} Let $C(n,k)$ denote the number of connected graphs with $n$ labeled vertices and $n+k-1$ edges. For example, $C(n,0)$ is the number of labeled trees on $n$ vertices and is equal to $n^{n-2}$ by Cayley's theorem. There is a rich history of the study of the asymptotics of the sequence $C(n,k)$. Wright gives the asymptotic formula when $k$ is fixed and $n\rightarrow \infty$ in \cite{W77} and when $k=o(\sqrt[3]{n})$ in \cite{W80}. Two different approaches were taken in analyzing the case when both $n,k\rightarrow \infty$, one by Bender, Canfield, and McKay \cite{brcm}, and the other by Coja-Oghlan, Moore, and Sanwalani \cite{coms}, and van der Hofstad and Spencer \cite{vdhs}. 

When $k=o(\sqrt{n})$ it has been observed that these limits (upto scaling) are also given by the moments of the area under a standard Brownian excursion. A standard Brownian excursion is a random element taking values in the subset of all nonnegative continuous functions on the interval $[0,1]$ (denoted by $C[0,1]$) given by 
\[
\left\{ \omega \in C[0,1]:\; \omega(0)=\omega(1)=0,\;\text{and}\; \omega(t) >0\; \text{for all $t \in (0,1)$}  \right\}.
\]
Informally, one can describe this process as a standard Brownian motion (starting from zero) conditioned to return to zero for the first time at time one. We refer the readers to the book by Revuz and Yor \cite{ry99} for a proper introduction. Another description of combinatorial interest is that a standard Brownian excursion is the contour of the Brownian Continuum Random Tree defined by Aldous \cite{A93}. In any case, the area under this random continuous curve is well-defined and measurable. Exact and asymptotic formulas for the moments of this random area can be found in the article by Louchard \cite{L} and the recent survey by Janson \cite{Jsurv}. 

Our main result is the following. 

\begin{thm}\label{thm:intro}
Let $A$ be a random variable whose law is given by the area under the standard Brownian excursion. Then, for all $(n,k)$ such that $k=k_n=o(\sqrt[3]{n})$ and $n\rightarrow \infty$ we have
\[
\lim_{n\rightarrow \infty} \frac{k!}{n^{n+3k/2-2}} \frac{C(n,k)}{ E A^k} = 1.
\]
\end{thm}

It can be found in \cite[Page 90, eqn.~(53)]{Jsurv} that
\eq\label{eq:momentasymp}
E A^k \sim 3\sqrt{2} k \left( \frac{k}{12 e} \right)^{k/2}, \qquad \text{as $k\rightarrow \infty$}. 
\en
Hence, the previous theorem reproves the result of \cite{coms} and \cite[Sec 2.1]{vdhs} (when $k=o(\sqrt[3]{n})$):
\eq\label{eq:cnkfromexc}
C(n,k) \sim n^{n-2} n^{3k/2} \left( e/12k \right)^{k/2}\left( 3/\sqrt{\pi} \right)k^{1/2}.
\en

The first explanation between the connection of $C(n,k)$ with the area under the standard Brownian excursion was given in a beautiful paper by Spencer \cite{spen97}. Let $Z_1, \ldots, Z_n$ be independent Poisson random variables with mean one. Let $Y_j=\sum_{i=1}^j Z_i$ denote the partial sum process. Define the queue walk by $Q_0=1$, and
\eq\label{eq:poissonsum}
Q_i= Q_{i-1} + (Z_i-1)=Y_i - (i-1), \quad i=1,2,\ldots, n. 
\en
Define the event $\exc:=\{Q_i>0, \; i=1,\ldots, n-1, \; \text{and}\; Q_n=0  \}$. Consider the empirical area
\[
M= \sum_{i=1}^{n-1} \left(Q_i -1\right)=\sum_{i=1}^{n-1} \left(Y_i - i\right). 
\]
Let $E^*$ denote the expectation conditioned on the event $\exc$. Then (see \cite[Theorem 3.2]{spen97}) the following \textit{exact} relationship holds
\eq\label{eq:spencer}
C(n,k)= n^{n-2} E^*\left[ \combi{M}{k}  \right].
\en

Now, under proper scaling, the law of the queue walk under $\exc$ converges in distribution to the law of a standard Brownian excursion. The area under the curve is a continuous function of the curve under the uniform distance. And hence, under proper scaling, which is dividing by $n^{3/2}$, $M$ converges in distribution to the area under the standard Brownian excursion. 

The original article by Wright \cite{W77}, had identified the limiting value of $C(n,k)$, when $k$ is fixed and as $n$ tends to infinity, as $n^{n-2} n^{3k/2} c_k/ k!$, where this sequence of constants $(c_k,\; k\in \NN)$ came to be known as Wright's constants.

Spencer observed, but did not prove, that if the weak convergence of $M$ to that of $A$ can be strengthened to a convergence of moments of fixed order, this would explain the factor of $n^{3k/2}$ in Wright's expression, and provide an interpretation of Wright's constants as the moment sequence of $A$. 

The limits of all sequences $C(n,k_n)$ have been derived in \cite{brcm} by analytic combinatorial methods, and in \cite{coms, vdhs} by probabilistic methods. For any sequence $k_n=o\left( \sqrt{n}\right)$, they are still given by the formula \eqref{eq:cnkfromexc}. However, these methods neither employ nor shed any light on the Brownian approximation.  

The main result in this article argues this in the regime of $k_n=o(\sqrt[3]{n})$. Our main tool is an explicit coupling that is made feasible by a strong approximation result due to Koml\'os, Major, and Tusn\'ady (KMT). This strong approximation result is the Brownian bridge version of the more famous KMT embedding that approximates a random walk by a Brownian motion. 

Interestingly, the proof breaks down beyond $k_n=o(\sqrt[3]{n})$ for technical reasons and cannot be improved by the currently known version of the KMT result. It is curious that this is the same regime that Wright \cite{W80} could extend his argument.

\section*{Acknowledgement} I am grateful to Prof.~Joel Spencer for suggesting the problem to me and several useful discussion. I also thank an anonymous referee for a careful reading of a previous version of the paper which corrected an earlier error. 

\section{Proof of the main result}

We will use the following notation throughout: For any event $A$, the random variable $1\{ A\}$ takes the value one if $A$ occurs, and takes the value zero otherwise. 

\begin{thm}\label{thm:main}
Consider the queue walk and the event $\exc$ from \eqref{eq:poissonsum}. Consider the rescaled continuous time process
\eq\label{whatisxt}
X_n(t)= \frac{1}{\sqrt{n-1}}\left(Q_{i}-1\right), \quad i/(n-1) \le t < (i+1)/(n-1), \quad i=0,\ldots n-2.
\en
Also define $X_n(1)=0$. Then $X_n$ is a path, defined on $[0,1]$, that starts and ends at zero. 
Consider the empirical area
\eq\label{eq:whatismn}
M_n:= \int_0^1 X_n(t)dt= \frac{1}{(n-1)^{3/2}} \sum_{i=0}^{n-1} \left( Q_i - 1 \right). 
\en
Let $E^*$ denote the expectation conditioned on the event $\exc$. Let $A$ be the area under a standard Brownian excursion. For all $k=k_n=o\left( \sqrt[3]{n} \right)$, we have
\[
\lim_{n\rightarrow \infty} \frac{E^*\left( M_n \right)^k}{ E A^k} =1.
\]
\end{thm}

\bigskip
Let us give an outline of the steps of the proof backwards in the order in which they appear. 
\begin{enumerate}
\item[Step 3.] We prove an embedding of the partial sum process, conditioned on $\exc$, in a standard Brownian excursion and estimate errors.  
\item[Step 2.] Step 3 follows by a known KMT embedding of the empirical process (to be defined later) in a Brownian bridge. 
\item[Step 1.] We describe how Step 3 follows from Step 2 by \textit{re-rooting at the minimum}. 
\end{enumerate}
We now expand each step of the proof. 

\subsection{Re-rooting at the minimum} Consider a deterministic sequence of numbers $x := (x_1, \ldots,  x_n)$. The walk with steps $x$ is the sequence of partial sums of $x$ starting at zero. That is
\[
s_0=0, \quad s_{i}= s_{i-1} + x_{i}, \quad i=1,2,\ldots
\]
For $i \in \{1,2,\ldots, n\}$ let $x^{(i)}$ denote the $i$th cyclic shift of $x$, that is the sequence of length $n$ whose $j$th term is $x_{i+j}$ where $i+j$ refers to $(i + j)\bmod n$. The following lemmas stem from the classical ballot problem and is used by Tak\'acs in \cite{T} to prove Kemperman's formula. For more details and the proofs see \cite[Section 6.1]{combstoc}.

\begin{lemma}\label{lemma:cyclicshift}
Let $x:=(x_1, \ldots, x_n)$ be a sequence with values in $\{ -1,0,1,\ldots\}$ and sum $-1$. Let $\sigma=\min\{i:\; s_i=\min_{1\le j\le n} s_j  \}$, i.e., the first time the walk reaches its absolute minimum. Then the walk with steps $x^{(\sigma)}$ hits $-1$ for the first time at step n. Moreover, this is the only $i$ such that the walk with steps $x^{(i)}$ hits $-1$ for the first time at step n.   
\end{lemma}

We say the path $S$ with steps $x$ has been re-rooted at the minimum to denote the path $R(S)$ with steps $x^{(\sigma)}$. The following lemma, called the discrete Vervaat's transform, follows from above and the exchangeability of increments \cite[page 125]{combstoc}.

\begin{lemma}\label{lemma:exchangeable}
Suppose $X=(X_1, \ldots, X_n)$ denote a sequence of iid random variables with values in $\{-1,0,1,\ldots\}$. Let $S_0=0$, and $S_{i+1}=S_i+X_{i+1}$, for $i=0,1,\ldots,n-1$. We refer to the sequence $S=(S_0, S_1, \ldots, S_n)$ as a path. Let $\bridge$ denote the event $S_n=-1$. Then, if $S$ is a path, conditioned on \bridge, then $R(S)$ is distributed as a path conditioned on $\exc=\{ S_n=-1,\; S_i \ge0, \; 1\le i<n  \}$. 
\end{lemma}

A continuous version of this can be easily guessed \cite[page 125]{combstoc}. Recall that the Brownian bridge is a standard BM conditioned to be zero at time one. 

\begin{lemma}\label{lemma:vervaat}
Let $\sigma$ be the (almost surely) unique time when a Brownian bridge $B$ achieves its global minimum. Then the process given by the cyclic transformation $\{  B_{\sigma+t} - B_\sigma, \; 0\le t \le 1\}$ (understood $\bmod 1$) is a standard Brownian excursion. 
\end{lemma}

To use these results for our problem we need to understand the event $\bridge$ for the queue walk. Consider a unit rate Poisson point process on the half line $[0,\infty)$. That is, consider iid Exponential random variables $\{ T_i, \; i \in \mathbb{N}\}$ with mean one, and consider their partial sums $S_0=0$, and $S_n = \sum_{i=1}^n T_i$, $n\ge 1$. We say an `event' occurs at each `time point' $S_n$. Then, one can define $Z_i$ to be the number of events occurred during the interval $[i-1, i)$, for $i\ge 1$. This maintains that $Z_1, Z_2, \ldots$ are iid Poisson random variables with mean one. 

The event $\{Q_n=0\}=\{Y_n=n-1\}$ is equivalent to the statement that there are $(n-1)$ events in time interval $[0,n]$. It is well-known, that conditioned on the number of events in a given interval, their points of occurrences, under the Poisson process, are independently and uniformly distributed. Thus, under $\bridge$, the times of occurrences of the $(n-1)$ events are given by $(n-1)$ iid Uniform$(0,1)$ random variables multiplied by $n$. 
  
Let $U_1, U_2, \ldots,$ denote a countable sequence of iid Uniform$(0,1)$ random variables. Consider the process
\eq\label{eq:whatisf}
F_{n-1}(t):= \sum^{n-1}_{i=1}1\left\{ U_i \le t  \right\} - nt, \quad t \in [0,1].
\en
Then the sequence $(F_{n-1}(k/n), \; k=0,1,2,\ldots,n)$ has the same law as $(Q_k-1,\; k=0,1,2,\ldots,n)$ under $\bridge$. Note that we have suppressed the dependence of $n$ from $Q_k$. In what follows the correct $n$ will be obvious from the context.
Also define what is known as the \textit{empirical process}
\eq\label{eq:whatisgn}
G_{n}(t):= \sqrt{n}\left( \frac{1}{n}\sum^{n}_{i=1}1\left\{ U_i \le t  \right\} - t\right) = \frac{1}{\sqrt{n}} \left( F_n(t) + t \right), \quad t\in [0,1]. 
\en

\subsection{Strong embedding for the empirical process.} The famous Koml\'os-Major-Tusn\'ady (KMT) paper \cite{kmt} contains a result on the strong embedding of the empirical process in the Brownian bridge. Several variations of the result have since been discovered due to its importance in the empirical process theory. See \cite{cchm}, \cite{mz87}. We take the statement of the following lemma from \cite{mz87}.  

\begin{lemma}\label{lemma:strong}
There is a probability space on which one can define random variables $U_1, U_2, \ldots$ that are iid Uniform $(0,1)$ and a sequence of processes $B_1, B_2, \ldots$ that are each distributed as a standard Brownian bridge such that if we define
\[
\Delta_1:=\Delta_1(n)= \sup_{0\le s \le 1}\abs{G_n(s) -  B_n(s)}, 
\]
then there exist universal positive constants $a,K,\lambda$ such that
\eq\label{eq:kmtbnd}
P\left( \sqrt{n} \Delta_1 > a \log n + x  \right) < K e^{- \lambda x}, \quad \text{for all}\; x\in [0, \infty). 
\en
\end{lemma}

Consider now a path of $G_n$ and $B_n$ as defined on the probability space above. We now have three processes: the process $X_{n+1}$ as defined in \eqref{whatisxt}, the empirical process $G_n$ and the Brownian bridge $B_n$. We know by Lemma \ref{lemma:strong} that the supremum distance between the processes $G_n$ and $B_n$ is $\Delta_1$. 

We now estimate the supremum distance between $X_{n+1}$ and $G_n$. Note from the line following \eqref{eq:whatisf}
\[
X_{n+1}(t)= \frac{1}{\sqrt{n}} F_n\left( \frac{k}{n+1} \right), \quad \frac{k}{n+1}\le t < \frac{k+1}{n+1},\quad k=0,1,\ldots,n.
\]
Subtracting, from \eqref{eq:whatisgn}, we get
\[
\abs{X_{n+1}(t) - G_n(t) } \le \frac{1}{\sqrt{n}} \abs{F_n(t) - F_n\left( \frac{k}{n+1} \right)}.
\]
Using \eqref{eq:whatisgn} again, for $k/(n+1) \le t < (k+1)/(n+1)$, we have 
\[
\abs{X_{n+1}(t) - G_n(t) } \le \abs{G_n(t) - G_n\left(\frac{k}{n+1} \right)} + \frac{1}{(n+1)\sqrt{n}}. 
\]

Taking supremum on both sides above and using Lemma \ref{lemma:strong} we get
\[
\begin{split}
\Delta_2&:=\sup_{0\le t\le 1} \abs{X_{n+1}(t) - G_n(t) } \\
&  \le \frac{1}{\sqrt[3]{n}} + \sup_{0\le k \le n}\sup_{k< (n+1) t < k+1} \abs{B_n(t) - B_n(k/(n+1))} + 2\Delta_1.
\end{split}
\]

Express the Brownian bridge $B_n$ as $(\beta_n(t) - t \beta_n(1),\; 0\le t\le 1)$, for some standard Brownian motion $\beta_n$ (see \cite[page 37]{ry99}). Then
\[
\begin{split}
\sup_{k/(n+1)<  t < (k+1)/(n+1)}& \abs{B_n(t) - B_n(k/(n+1))} \le \\
&\sup_{k/(n+1)< t < (k+1)/(n+1)} \abs{\beta_n(t) - \beta_n(k/(n+1))} + \frac{\abs{\beta_n(1)}}{n+1}.
\end{split}
\]
By the Markov property of Brownian motion we know that for each $k$, the quantity 
\[
\abs{Z_k}/\sqrt{n+1}:=\sup_{k/(n+1)\le t\le (k+1)/(n+1)}\abs{\beta_n(t) - \beta_n(k/(n+1))}
\]
are independent and identically distributed. In fact, by the stationary increment property of Brownian motion and scaling, this distribution is the same as that of $\abs{\beta_n}^*=\sup_{0\le s \le 1} \abs{\beta_n(s)}$. Moreover, by Paul L\'evy's Characterization Theorem (\cite[page 240]{ry99}), we know that $\overline{\beta_n}:=\sup_{0\le s \le 1} \beta_n(s)$ and $\underline{\beta_n}:=\sup_{0\le s\le 1} -\beta_n(s)$ are both distributed as the absolute value of a standard Normal random variable $N$. Observe that, for any positive $x$,
\[
P(\abs{N} > x)= P(\overline{\beta_n} > x) \le   P\left( \abs{\beta_n}^* > x \right) \le P(\overline{\beta_n} > x) + P\left( \underline{\beta_n} > x \right)= 2P(\abs{N} > x).
\]
Thus the distribution of each $Z_k$ (and all their moments) are comparable to that of the absolute value of the standard Normal. We will use this fact implicitly in the following argument. In any case
\[
\Delta_2 \le \frac{1}{\sqrt[3]{n}} + \frac{1}{\sqrt{n+1}} \max_{0\le i \le n} \abs{Z_i} + \frac{\abs{\beta_n(1)}}{n+1}+ 2 \Delta_1. 
\]

Hence, the supremum distance between the continuous time walk $X_{n+1}$ and the Brownian bridge 
\eq\label{eq:supest}
\Delta_{(n)}:= \sup_{0\le t \le 1}\abs{X_{n+1}(t) - B_n(t)}\le \Delta_1 + \Delta_2=  \frac{1}{\sqrt[3]{n}} + \frac{1}{\sqrt{n+1}} \max_{0\le i \le n} \abs{Z_i} + \frac{\abs{\beta_n(1)}}{n+1}+ 3 \Delta_1. 
\en
\bigskip

\begin{figure}[t]
\centering
\includegraphics[width=5in, height=1.6in]{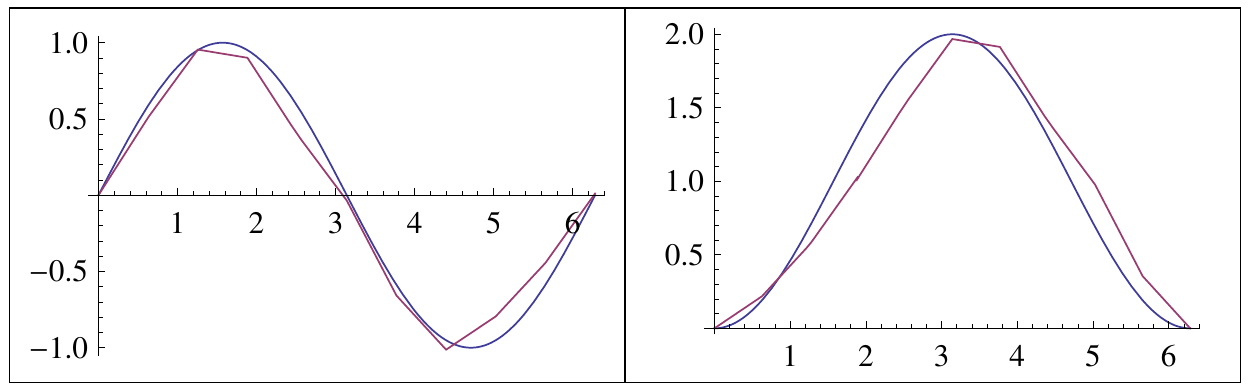}
\caption{Re-rooting at the minimum}
\label{fig:reroot}
\end{figure}

We now re-root at the minimum both the continuous time walk $X_{n+1}$ and the Brownian bridge $B_n$. Please see Figure \ref{fig:reroot} where the smooth curve is the sine curve which is approximated by a jagged path. By \eqref{eq:supest} their minimums differ by $\Delta_{(n)}$ and can be attained at different times $\sigma^{w}$ (for the walk) and $\sigma^{b}$ (for the Brownian bridge). Depending on how close $\sigma_w$  and $\sigma_b$ are, after re-rooting we get a walk excursion and a standard Brownian excursion which might be a little off. However, this does not affect the total area. 

To see what we mean, suppose $y$ is a continuous curve on $[0,1]$ such that $y(0)=y(1)=0$ and with an absolute minimum $y_{\min}$. Then the area under the curve $\tilde{y}(\cdot):= y(\cdot) + y_{\min}$ is equal to the area under the curve that is obtained by re-rooting $y$ at its minimum. The difference between the minimums of the walk and the Brownian bridge is at most $\Delta_{(n)}$, which is also the uniform distance between the two curves. Thus the area between the two re-rooted curves differ by at most $2 \Delta_{(n)}$. We have the following lemma.

\begin{thm}\label{thm:space}
Consider the set-up in Theorem \ref{thm:main}. On the KMT space described in Lemma \ref{lemma:strong}, it is possible to have for each $n\in \NN$ a copy of $M_n$, under $\exc$, and a random variable $A_n$, distributed according to the area of a standard Brownian excursion such that if $D_n=\abs{M_n - A_n}$, then there exists two absolute positive constants $C_1$ and $C_2$ such that for all $k,n \in \NN$ we have
\[
\left[ E D_n^k\right]^{1/k} \le C_1 \left( \frac{\log n}{\sqrt{n}} \right) + C_2 \left( \frac{k}{\sqrt{n}}\right), 
\]

\end{thm}

\begin{proof}[Proof of Theorem \ref{thm:space}] From \eqref{eq:supest}, the discussion above, and by using the triangle inequality we get
\eq\label{eq:triangle}
\begin{split}
\left(E D_n^{k}\right)^{1/k} &\le  2 \left[ E\left( \frac{1}{\sqrt[3]{n}} + \frac{1}{\sqrt{n+1}} \max_{0\le i \le n} \abs{Z_i} + \frac{\abs{\beta_n(1)}}{n+1}+ 3 \Delta_1   \right)^k\right]^{1/k}\\
&\le 2\left[ \frac{1}{\sqrt[3]{n}}   + \frac{1}{\sqrt{n}} \left( E\left(\max_{i} \abs{Z_i}\right)^k\right)^{1/k} +\frac{1}{n+1}\left(E\abs{\beta_n(1)}^k\right)^{1/k}+ 3 \left(E \Delta_1^k\right)^{1/k}   \right],\\
&\le \frac{2}{\sqrt[3]{n}} + \frac{2}{\sqrt{n}} \left[ E\left(\max_{i} \abs{Z_i}\right)^k\right]^{1/k} + \frac{2}{n+1}\left(E\abs{\beta_n(1)}^k\right)^{1/k}+ 6 \left( E \Delta_1^k\right)^{1/k}.
\end{split}
\en

Let $\mu_n=E\max_i \abs{Z_i}$. It is well-known that $\mu_n=O(\sqrt{\log n})$. From the Gaussian concentration of measure we know that the probability density of $\abs{\max_i \abs{Z_i} - \mu_n}$ has a sub-Gaussian tail with variance $1$. Thus, from known moments of the standard Normal distribution we infer that
\[
\begin{split}
\left(E\left(\max_{i} \abs{Z_i}\right)^k\right)^{1/k} &\le \mu_n + \left( E\abs{\max_i \abs{Z_i} - \mu_n}^k\right)^{1/k} \le \mu_n + C^* \sqrt{k},\\
\left(E\abs{\beta_n(1)}^k\right)^{1/k} &\le C^* \sqrt{k},
\end{split}
\]
where $C^*$ is some absolute constant. 

Further, note that $\Delta_1=\sup_s\abs{G_n(s) - B_n(s)}$. Now
\[
\sqrt{n} \Delta_1 \le a\log n + \left(\Delta_1 - a \log n  \right)^+,
\] 
where $x^+= x1\{ x>0\}$. Again using the triangle inequality for the $k$th norm and the bound from \eqref{eq:kmtbnd}, we get
\[
\begin{split}
\sqrt{n} \left( E \Delta_1^k\right)^{1/k} &\le a \log n + \left[ k K \int_0^\infty x^{k-1} e^{-\lambda x} dx\right]^{1/k}= a \log n +  K^{1/k} \lambda^{-1} (k!)^{1/k}.
\end{split} 
\]
Thus
\[
\left[  E \Delta_1^k \right]^{1/k} \le  \frac{a\log n}{\sqrt{n}} +  \frac{O(k)}{\sqrt{n}}.
\]
Substituting in \eqref{eq:triangle} we get our desired result. 
\end{proof}

\begin{proof}[Proof of Theorem \ref{thm:main}]
From the last theorem we know that 
\eq\label{eq:poitoexp}
\left(  E D_n^{k} \right)^{1/k} \le C_1\left( \frac{\log n}{\sqrt{n}} \right) + C_2 \left( \frac{k}{\sqrt{n}} \right). 
\en
From here one can estimate the deviation of moments. By elementary calculus, 
\eq\label{eq:calcul}
 \abs{ (1+\epsilon)^k - 1} \le 2k \abs{\epsilon}, \quad \text{for all}\;  \abs{\epsilon} \le \frac{1}{2k}.
\en

Choose $\epsilon$ such that 
\[
(1+\epsilon)= \left[ E\left(  M_n \right)^{k}\right]^{1/k} / \left[ E\left( A_n \right)^{k} \right]^{1/k}.
\]
Clearly then
\[
\abs{\epsilon}\le \frac{\abs{\left[ E\left(  M_n \right)^{k}\right]^{1/k} -  \left[ E\left( A_n \right)^{k} \right]^{1/k}} }{\left[ E\left( A_n \right)^{k} \right]^{1/k}}\le \frac{\left[E \left(D_n\right)^k\right]^{1/k}}{\left[ E\left( A_n \right)^{k} \right]^{1/k}}.
\]

Note that $A_n$, being the area under a standard Brownian excursion, has a law that does not depend on $n$. From \eqref{eq:momentasymp} we claim the existence of another absolute constant $C_3 >0$ such that for all $k \in \NN$, 
\[
\left[ E\left( A_n \right)^{k} \right]^{1/k} \ge C_3 \sqrt{k}. 
\]
Thus
\[
\abs{\epsilon}\le O\left(\frac{\log n}{\sqrt{kn}}  \right) + O\left( \sqrt{\frac{k}{n}} \right). 
\]

Thus $k\epsilon$ converges to zero for all sequences such that $k=o(\sqrt[3]{n})$. This finishes the proof. 
\end{proof}

\begin{proof}[Proof of Theorem \ref{thm:intro}]
To finish the proof of Theorem \ref{thm:intro} we simply need to argue that when $k=o(\sqrt[3]{n})$, then 
\[
\lim_{n\rightarrow \infty} k!\frac{E^*\left[ \combi{M}{k} \right]}{E^*\left(M^k\right)}=1
\]
and use \eqref{eq:spencer}. However, this follows from elementary bounds since $M=O(n^{3/2})$ and $k=o(\sqrt[3]{n})$. We skip the details.
\end{proof}

\bibliographystyle{alpha}

\end{document}